\pdfoutput=1
\documentclass[11pt, a4paper]{amsart}
\usepackage[paperheight=279mm,paperwidth=18cm,textheight=26cm,textwidth=14cm,includehead]{geometry}
\usepackage{amsmath,amsthm,amssymb}
\usepackage{mathtools}
\usepackage{microtype}
\usepackage[T1]{fontenc}
\usepackage[utf8]{inputenc}
\usepackage[unicode]{hyperref}
\hypersetup{
bookmarks=true,
colorlinks=true,
citecolor=[rgb]{0,0,0.5},
linkcolor=[rgb]{0,0,0.5},
urlcolor=[rgb]{0,0,0.75},
pdfpagemode=UseNone,
pdfstartview=FitH,
pdfdisplaydoctitle=true,
pdftitle={Ap–A∞ estimates for multilinear maximal and sparse operators},
pdfauthor={Pavel Zorin-Kranich},
pdflang=en-US
}

\usepackage[style=alphabetic]{biblatex}
\numberwithin{equation}{section}
\newtheorem{theorem}[equation]{Theorem}
\newtheorem{corollary}[equation]{Corollary}
\newtheorem{lemma}[equation]{Lemma}
\theoremstyle{definition}

\newtheorem{definition}[equation]{Definition}
\newtheorem{proposition}[equation]{Proposition}

\newcommand{\R}{\mathbb{R}}
\newcommand{\Z}{\mathbb{Z}}
\newcommand{\dif}{\mathrm{d}}
\newcommand{\sparse}{\mathcal{S}}
\newcommand{\cM}{\mathcal{M}}
\newcommand{\cF}{\mathcal{F}}
\newcommand{\cD}{\mathcal{D}}
\newcommand{\sop}{\mathcal{A}}
\newcommand{\sform}{\mathcal{B}}
\newcommand{\Car}{\mathrm{Car}}
\newcommand{\one}{\mathbf{1}}
\newcommand*{\widevec}[1]{\overrightarrow{#1}}
\DeclarePairedDelimiter\abs{\lvert}{\rvert}
\DeclarePairedDelimiter\meas{\lvert}{\rvert}
\DeclarePairedDelimiter\norm{\lVert}{\rVert}
\DeclarePairedDelimiter\Set\{\}

\begin{document}

\title[Multilinear $A_{p}$--$A_{\infty}$ estimates]{$A_{p}$--$A_{\infty}$ estimates for multilinear\\ maximal and sparse operators}
\author{Pavel Zorin-Kranich}
\date{}
\address{Mathematical Institute\\ University of Bonn}
\thanks{This work was partially supported by the Hausdorff Center for Mathematics (DFG EXC 59)}
\maketitle
\begin{abstract}
We obtain mixed $A_{p}$--$A_{\infty}$ estimates for a large family of multilinear maximal and sparse operators.
Operators from this family are known to control for instance multilinear Calder\'on--Zygmund operators, square functions, fractional integrals, and the bilinear Hilbert transform.
Our results feature a new multilinear version of the Fujii--Wilson $A_{\infty}$ characteristic that allows us to recover the best known estimates in terms of the $A_{p}$ characteristic for dependent weights as a special case of the mixed characteristic estimates for general tuples of weights.
\end{abstract}

\section{Introduction}
\subsection{Relation to previous work}
The problem of optimal dependence of constants in weighted inequalities on characteristics of the weights has been studied since the introduction of the subject by Muckenhoupt, who has addressed this question for the weak type inequalities for the Hardy--Littlewood maximal operator in his seminal paper \cite{MR0293384}.
A characterization of pairs of weights for which the maximal operator is bounded from $L^{p}(w_{1})$ to $L^{p}(w_{2})$ has been obtained by Sawyer \cite{MR676801}, but the asymptotically sharp dependence of the operator norm on the $A_{p}$ Muckenhoupt characteristic of the weight in the one-weight situation $w_{1}=w_{2}$ has been only obtained much later by Buckley \cite{MR1124164}.
Similar questions for singular integral operators have been resolved by a number of authors, culminating in Hytönen's $A_{2}$ theorem \cite{MR2912709}.
We refer to \cite{MR3204859} for a detailed account of these developments.

Mixed $A_{p}$--$A_{\infty}$ characteristic weighted estimates of the type that we are interested in have been first obtained in \cite{MR3092729}.
In particular, \cite[Theorem 1.10]{MR3092729} refines Buckley's estimate for the maximal function and \cite[Theorem 1.3]{MR3092729} refines Hytönen's $A_{2}$ theorem.
We do not state these results here because we are interested in their multilinear versions, and it seems more natural to discuss them in more detail after defining the relevant objects in the multilinear setting.

The proof of the $A_{2}$ theorem has been greatly simplified since its original appearance.
One of the main contributors to this effort has been Lerner, who has been able to estimate singular integral operators by a certain type of positive operators, called \emph{sparse operators}, see e.g.\ \cite{MR3085756}, and therefore reduce the problem of proving weighted estimates to the sparse operators.
An even simpler way to make this reduction has been later discovered by Lacey \cite{arXiv:1501.05818} and further refined in \cite{MR3625128,MR3484688,arXiv:1604.05506}.

Therefore we concentrate on weighted estimates for sparse operators.
The observation that the $A_{2}$ bound is very easy to prove for (linear) sparse operators seems to have been first made in \cite{MR2628851} and extended to $A_{p}$ in \cite{MR3000426}.
Mixed $A_{p}$--$A_{\infty}$ two-weight estimates for linear sparse operators and bilinear sparse forms that generalize these results have been obtained in \cite{MR3778152} and \cite{MR3591468}.

Weighted estimates for multilinear sparse forms with dependent weights (generalizing the one-weight inequalities for (sub)linear operators) have been obtained in \cite[Theorem 1.2]{MR3232584} and \cite[\textsection 16]{arXiv:1508.05639}.
For independent weights (corresponding to two-weight inequalities for (sub)linear operators), certain mixed $A_{p}$--$A_{\infty}$ estimates have been proved in \cite{arXiv:1508.06105,MR3850284}, however, they do not recover the previous results for dependent weights.
Our objective is to obtain such mixed estimates for independent weights that do recover the best known power of multilinear $A_{p}$ characteristic in the case of dependent weights.

We also obtain corresponding results for multi(sub)linear maximal functions (introduced in \cite{MR2483720}) which extend the sharp bounds in the case of dependent weights \cite[Theorem 1.2]{MR3232584}, refining some previous mixed characteristic bounds \cite[Theorem 3.6]{MR3118310} that do not.
Finally, our results are formulated in a way that makes them applicable also to fractional maximal functions and fractional integrals along the lines of \cite{MR3224572}.
We refer to \cite{MR3000426} and \cite{MR3483175} for the estimates of fractional integrals by sparse forms in the linear and the multilinear setting, respectively, and to \cite{MR3688143} for more weighted inequalities for fractional integrals.

We do not address the question of characterizing the families of weights for which our operators are bounded.
For recent work in the spirit of Sawyer's testing conditions for the maximal operator \cite{MR676801} see \cite{arXiv:0911.3437} and \cite{arXiv:1503.06778}, and in the multilinear setting \cite{MR3554405} and \cite{MR3490550}.

\subsection{Maximal functions and weights}
We will work in a non-atomic measure space $(X,\mu)$.
We will write $\meas{E}=\mu(E)$ for subsets $E\subset X$.
A \emph{dyadic grid} is a collection $\cD$ of measurable subsets of $X$ such that for any $Q,Q'\in\cD$ we have either $Q\subseteq Q'$ or $Q'\subseteq Q$ or $Q\cap Q'=\emptyset$.
We will call the members of $\cD$ ``cubes'' because we are mostly interested in the case of them being the standard dyadic cubes in $\R^{n}$.
Note however that we do not make any doubling assumption on the measure $\mu$, and our results apply to general discrete martingales in the setting of \cite{MR3406523} and \cite[\textsection 2]{arXiv:1501.05818}.

Since we only consider positive operators, all functions will be assumed to be positive in order to simplify notation.

\begin{definition}
Let $\cD$ be a dyadic grid, $m\geq 2$, $0<r_{i}<\infty$, and $0\leq \rho_{i}<1$ for $1\leq i<m$.
The \emph{multilinear martingale fractional maximal function} is defined by
\begin{equation}
\label{eq:mult-frac-max}
\cM^{\vec r}_{\vec \rho}(\vec f)(x)
:=
\sup_{x\in Q, Q\in\cD} \prod_{i=1}^{m-1} \Big(\meas{Q}^{-(1-\rho_{i})} \int_{Q} f_{i} \Big)^{r_{i}}.
\end{equation}
\end{definition}
We will omit the index $\vec\rho$ if $\vec\rho=(0,\dots,0)$ and the index $\vec r$ if $\vec r=(1,\dots,1)$.
In particular, $\cM$ is the martingale version of the multilinear maximal operator defined in \cite[Definition 3.1]{MR2483720}.
In the case $m=2$ we write $M^{r}_{\rho} = \cM^{\{r\}}_{\{\rho\}}$ with the same conventions for index omission.
In particular, $M_{\rho}$ is the martingale fractional maximal operator as in \cite{MR3000426}.
We will sometimes indicate the reference measure $\mu$ by a subscript to the maximal function.

\begin{definition}
A \emph{weight} is a non-negative measurable function on $X$.
Let $m\geq 2$ and $\vec w=(w_{1},\dots,w_{m})$ be a tuple of weights.
For arbitrary exponents $0\leq q_{1},\dots,q_{m}<\infty$ the \emph{multilinear Muckenhoupt characteristic} is defined by
\begin{equation}
\label{eq:def:Muck}
[\vec w]^{\vec q} := \sup_{Q} \prod_{i=1}^{m} (w_{i})_{Q}^{q_{i}} = \norm{\mathcal{M}^{\vec q}(\vec w)}_{\infty},
\end{equation}
where $(w)_{Q} := \meas{Q}^{-1} \int_{Q} w$ is the average of $w$ over $Q$,
and the \emph{multilinear Fujii--Wilson (FW) characteristic} by
\begin{equation}
\label{eq:def:FW}
[\vec w]_{FW}^{\vec q} := \sup_{Q} \Big( \int_{Q} \mathcal{M}^{\vec q/q} (1_{Q} \vec w) \Big)^{q} \Big( \int_{Q} \prod_{i=1}^{m} w_{i}^{q_{i}/q} \Big)^{-q},
\quad
q = \sum_{i=1}^{m} q_{i}.
\end{equation}
\end{definition}
The above way of defining the multilinear Muckenhoupt characteristic \eqref{eq:def:Muck} has been introduced in \cite{arXiv:1508.05639}, although with a different notation.
Our notation has the advantage that
\[
([w]^{\vec q})^{\alpha} = [w]^{\alpha \vec q}
\quad\text{and}\quad
([w]_{FW}^{\vec q})^{\alpha} = [w]_{FW}^{\alpha \vec q}
\]
holds for all $\vec q$ and $0\leq \alpha<\infty$.

For $m=1$ our FW characteristic coincides with (the $q_{1}$-th power of) the $A_{\infty}$ characteristic originating in the work of Fujii \cite{MR0481968} and Wilson \cite{MR883661}, which is in turn smaller than the $A_{\infty}$ characteristic based on the logarithmic maximal function originating in the work of Garc\'ia-Cuerva and Rubio de Francia \cite{MR807149} and Hru\v{s}\v{c}ev \cite{MR727244} (the latter observation has been first made in \cite{MR3092729}).
For $m\geq 2$ our version of the FW characteristic is smaller than the version introduced in \cite{MR3118310} and also smaller than the $H^{\infty}$ characteristic in \cite{MR3850284} (the latter being a consequence of $L^{1}$ boundedness of the logarithmic maximal function, see \cite[Lemma 2.1]{MR3092729}).
A key advantage of our FW characteristic is that, in the case of dependent weights, it can be estimated by a nice power of the multilinear Muckenhoupt characteristic:
\begin{lemma}
\label{lem:FW-Ap}
Let $m\geq 2$, weights $w_{1},\dots,w_{m}$, and exponents $0<q_{1},\dots,q_{m}<\infty$ be such that $\prod_{i=1}^{m} w_{i}^{q_{i}} = 1$.
Then for every choice of exponents $0\leq \beta_{1},\dots,\beta_{m}$ we have
\[
[\vec w]_{FW}^{\vec\beta} \lesssim [\vec w]^{\gamma \vec q}
\quad\text{with}\quad
\gamma = \max_{i} \beta_{i}/q_{i}.
\]
\end{lemma}
Lemma~\ref{lem:FW-Ap} is a generalization of the one-weight inequality $[w]_{A_{\infty}} \lesssim_{p} [w]_{A_{p}}$, $1<p<\infty$, which can be recovered with $m=2$, $\vec w=(w,w^{1-p'})$, $\vec\beta=(1,0)$, and $\vec q = (1,p-1)$.
It is proved in Section~\ref{sec:dependent}.

Our first results are weak and strong type estimates for the multilinear fractional maximal function.
Here and later we write $(\beta_{i})_{i\neq j}$ for the vector whose $i$-th entry is $\beta_{i}$ for $i\neq j$ and $0$ for $i=j$.
\begin{theorem}
\label{thm:max}
Let $m\geq 2$, $0<r_{i}<\infty$, $0\leq \rho_{i}<1$, $1<t_{i}<1/\rho_{i}$ for $i=1,\dots,m-1$ and let $\alpha := \sum_{i=1}^{m-1} r_{i}(1/t_{i}-\rho_{i})$.
Then
\begin{equation}
\label{eq:max-weak}
\norm{ \cM^{\vec r}_{\vec\rho}(\widevec{fw}) }_{L^{1/\alpha,\infty}(w_{m})}
\lesssim
[\vec w]^{\vec q} \prod_{i=1}^{m-1} \norm{ f_{i} }_{L^{t_{i}}(w_{i})}^{r_{i}}
\end{equation}
and
\begin{equation}
\label{eq:max-strong}
\norm{ \cM^{\vec r}_{\vec\rho}(\widevec{fw}) }_{L^{1/\alpha}(w_{m})}
\lesssim
[\vec w]^{\vec q} [\vec w]_{FW}^{(r_{i}(1/t_{i}-\rho_{i}))_{i\neq m}}
\prod_{i=1}^{m-1} \norm{ f_{i} }_{L^{t_{i}}(w_{i})}^{r_{i}}.
\end{equation}
where $q_{i} = r_{i}/t_{i}'$ for $i<m$ and $q_{m}=\alpha$.
\end{theorem}
The weak type estimate \eqref{eq:max-weak} is a straightforward generalization of \cite[Theorem 1]{MR0293384}, see Section~\ref{sec:weak-max} for the proof.
The strong type estimate \eqref{eq:max-strong} has some new aspects even in the non-fractional case $\vec\rho=0$.
In the linear case $m=2$, $\vec\rho=0$, the inequality \eqref{eq:max-strong} reduces to the above-mentioned \cite[Theorem 1.10]{MR3092729} (a different proof appeared in \cite[Theorem 1.3]{MR2990061}).
In the multilinear case the inequality \eqref{eq:max-strong} is a joint strengthening of \cite[Theorem 1.2]{MR3232584} and \cite[Theorem 3.6]{MR3118310} (the latter in turn being a refinement of \cite[Theorem 1.1]{MR3302105}).

\subsection{Carleson sequences and sparse forms}
Let $\cD$ be a dyadic grid.
A map $\tau:\cD\to [0,\infty)$, $Q\mapsto \tau_{Q}$, is called a \emph{Carleson sequence} if
\[
\norm{\tau}_{\Car}:=
\sup_{Q\in\cD} \frac{1}{\meas{Q}} \sum_{Q'\in\cD, Q'\subseteq Q} \tau_{Q'} \meas{Q'} < \infty.
\]
It is customary to consider $\{0,1\}$-valued Carleson sequences, in which case the collection of cubes on which $\tau$ takes the value $1$ is said to be \emph{$\Lambda$-Carleson} with $\Lambda=\norm{\tau}_{\Car}$.
A collection $\sparse$ of cubes if called \emph{$\eta$-sparse} if there exist pairwise disjoint subsets $E(Q)\subset Q$, $Q\in\sparse$, such that $\meas{E(Q)}\geq \eta \meas{Q}$.
A family of cubes is $\Lambda$-Carleson if and only if it is $1/\Lambda$-sparse, see \cite[Lemma 6.3]{arXiv:1508.05639}.
A similar fact holds for Carleson sequences with an identical proof:
\begin{lemma}
\label{lem:carleson-sparse}
Let $\cD$ be a dyadic grid on a non-atomic measure space $X$, $\Lambda>0$, and let $\tau$ be a Carleson sequence with $\norm{\tau}_{\Car}\leq\Lambda$.
Then there exist pairwise disjoint subsets $E(Q)\subset Q$, $Q\in\cD$, such that $\meas{E(Q)}=\tau_{Q}\meas{Q}/\Lambda$.
\end{lemma}
Since in our applications it suffices to consider finite dyadic grids and the proof of \cite[Lemma 6.3]{arXiv:1508.05639} is particularly easy in this case, we include it in Section~\ref{sec:key} for completeness.
Throughout the article we consider Carleson sequences of the form
\[
\tau_{Q} = \one_{\sparse}(Q),
\quad
\sparse \text{ sparse},
\]
with a uniform bound on the sparsity/Carleson constants.
The restriction to Carleson sequences coming from sparse sets is needed in the proof of Lemma~\ref{lem:key}.

We consider the following sparse forms associated to a Carleson sequence $\tau$.
Let $m\geq 2$, $0<r_{i}<\infty$, and $0\leq \rho_{i}<1$ for $1\leq i\leq m$.
Define the multi(sub)linear sparse operator by
\[
\sop_{\vec\rho}^{\vec r}(\tau,\vec f)
:=
\sum_{Q\in\cD} \tau_{Q} \prod_{i=1}^{m-1} \Big(\meas{Q}^{-(1-\rho_{i})} \int_{Q} f_{i} \Big)^{r_{i}} 1_{Q}
\]
and the multi(sub)linear sparse form by
\[
\sform_{\vec\rho}^{\vec r}(\tau,\vec f)
:=
\sum_{Q\in\cD} \tau_{Q} \meas{Q} \prod_{i=1}^{m} \Big(\meas{Q}^{-(1-\rho_{i})} \int_{Q} f_{i} \Big)^{r_{i}}.
\]
We apply to $\sop_{\vec\rho}^{\vec r}$ and $\sform_{\vec\rho}^{\vec r}$ the same index omission conventions as to $\cM_{\vec\rho}^{\vec r}$.
The sparse operators $\sop$ are known to dominate many multilinear singular integral operators, see \cite{MR3232584}, \cite{MR3850284}, \cite{arXiv:1604.05506}.
The sparse operators $\sop_{\vec\rho}$ with $\rho_{i}\neq 0$ dominate multilinear fractional integrals, see \cite{MR3483175}.
Finally, the sparse forms $\sform^{\vec r}$ with $r_{i}\neq 1$ turn out to dominate some singular integral operators with weak kernel regularity \cite{MR3802254} and even operators that fall outside the scope of Calder\'on--Zygmund theory \cite{MR3531367}, \cite{arXiv:1603.05317}, \cite{MR3653057}.

In the duality range we have the following estimate for the sparse forms $\sform_{\vec \rho}^{\vec r}$.
\begin{theorem}
\label{thm:fractional-dualized}
Let $m\geq 2$, $0<r_{i}<\infty$, $0\leq \rho_{i}<1$, $1<t_{i}<1/\rho_{i}$ and suppose
\[
\sum_{i=1}^{m} r_{i} (1/t_{i} - \rho_{i}) = 1.
\]
Let $\{1,\dots,m\}=J_{s}\cup J_{r}$ be a partition with $J_{s}$ non-empty.
Then
\[
\sform_{\vec\rho}^{\vec r}(\tau,\widevec{fw})
\lesssim
[\vec w]^{(r_{i}/t_{i}')_{i}} \sum_{j\in J_{s}} [\vec w]_{FW}^{(r_{i}(1/t_{i}-\rho_{i}))_{i\neq j}}
\prod_{i\in J_{s}} \norm{ f_{i} }_{L^{t_{i}}(w_{i})}^{r_{i}}
\prod_{i\in J_{r}} \norm{ f_{i} }_{L^{t_{i},r_{i}}(w_{i})}^{r_{i}}.
\]
\end{theorem}
In the case $m=2$ this recovers \cite[Theorem 1.6]{MR3591468} and both weak and strong type estimates in \cite{MR3778152} in the case $p>r$.
With $r_{i}=1$ and $J_{r}=\emptyset$ this recovers the case $p>\gamma$ of \cite[Theorem 4.2]{MR3850284}.
Further results in the off-diagonal case (not assuming the H\"older type condition on the exponents) appear in \cite{MR3761937}.

In the case of dependent weights $\prod_{i=1}^{m} w_{i}^{q_{i}}=1$ (that specializes to one weight inequalities in the ``linear'' case $m=2$), strong type bounds ($J_{r}=\emptyset$) have been obtained in \cite[Theorem 1.2]{MR3232584} for the form $\sform$ and in \cite[\textsection 16]{arXiv:1508.05639} for the form $\sform^{\vec r}$.
The dependence on the characteristic of the weights in these works takes the form
\[
[\vec w]^{\gamma (r_{i}/t_{i}')_{i}},
\qquad
\gamma = \max_{i} t_{i}'.
\]
These results can be recovered from Theorem~\ref{thm:fractional-dualized} using Lemma~\ref{lem:FW-Ap} and the fact that $\frac{r_{i}}{t_{i}} \frac{t_{i}'}{r_{i}} = t_{i}'-1$.

Outside of the duality range we have the following result, which extends the remaining cases of \cite[Theorem 1.1]{MR3778152} and \cite[Theorem 4.2]{MR3850284}.
\begin{theorem}
\label{thm:fractional-below-L1}
Let $m\geq 2$, $0<r_{i}<\infty$, $0\leq \rho_{i}<1$, $1<t_{i}<1/\rho_{i}$ and suppose
\[
\alpha := \sum_{i=1}^{m-1} r_{i} (1/t_{i} - \rho_{i}) \geq 1.
\]
Then
\[
\norm{ \sop_{\vec\rho}^{\vec r}(\tau,\widevec{fw}) }_{L^{1/\alpha}(w_{m})}
\lesssim
[\vec w]^{\vec q} [\vec w]_{FW}^{(r_{i}(1/t_{i}-\rho_{i}))_{i\neq m}}
\prod_{i=1}^{m-1} \norm{ f_{i} }_{L^{t_{i}}(w_{i})}^{r_{i}}.
\]
where $q_{i} = r_{i}/t_{i}'$ for $i<m$ and $q_{m}=\alpha$.
\end{theorem}
The case of (sub-)linear operators $m=2$ treated in \cite[Theorem 1.2]{MR3778152} suggests that the dependence on the FW characteristic in the corresponding weak type estimate can be removed, analogously to the estimate for the multilinear maximal operator.
We have been unable to decide whether this is indeed possible.

\subsection*{Acknowledgment}
I thank Bas Nieraeth for pointing out an error in a previous version of this article.

\section{Deduction of weighted estimates}
\subsection{Preliminaries}
Lebesgue space bounds for the martingale fractional maximal function can be found in \cite[Theorem 2.3]{MR3000426}.
By real interpolation they imply the following Lorentz space bounds.
\begin{proposition}
\label{prop:frac-max}
Let $0\leq\rho<1$, $1<p<1/\rho$, and let $1/q = 1/p - \rho$.
Then for every $1\leq r\leq\infty$, every dyadic grid $\cD$, and every measure $\nu$ we have
\[
\norm{M_{\rho,\nu} f}_{L^{q,r}(d\nu)} \lesssim_{\rho,p,r} \norm{f}_{L^{p,r}(d\nu)},
\]
and the implicit constant does not depend on the measure $\nu$ and the dyadic grid $\cD$.
\end{proposition}

Next, we include the proof of equivalence between Carleson and sparseness conditions in the special case of finite dyadic grids, cf.\ \cite[\textsection 6]{arXiv:1508.05639}.
\begin{proof}[Proof of Lemma~\ref{lem:carleson-sparse} for finite dyadic grids]
We induct on the cardinality of $\cD$.
If $\cD$ is empty, there is nothing to prove.

Let now $\cD$ be given and suppose that the result is known for all dyadic grids of smaller cardinality.
Let $Q\in\cD$ be a maximal element with respect to inclusion, then the restriction of $\tau$ to $\cD':=\cD\setminus\{Q\}$ also has Carleson norm $\leq \Lambda$.
By the inductive hypothesis there exist disjoint sets $E(Q')\subset Q'$, $Q'\in\cD'$, satisfying the conclusion of the lemma.
By the Carleson measure hypothesis we have
\[
\sum_{Q'\subsetneq Q} \meas{E(Q')}
=
\Lambda^{-1} \sum_{Q'\subsetneq Q} \tau_{Q'}\meas{Q'}
\leq
\Lambda^{-1} (\norm{\tau}_{\Car}-\tau_{Q})\meas{Q}
\leq
(1-\tau_{Q}/\Lambda)\meas{Q}.
\]
Therefore, since $X$ is non-atomic, there exists a subset $E(Q)\subset Q$ with measure $\tau_{Q}\meas{Q}/\Lambda$ that is disjoint from all $E(Q')$'s.
\end{proof}

We will use the following result about $L^{s}$ norms that is part of \cite[Proposition 2.2]{MR2086703}.
\begin{lemma}
\label{lem:cov}
For every $1<s<\infty$ there exists $C_{s}>0$ such that for every positive locally finite measure $\sigma$ on $X$ and any positive numbers $\lambda_{Q}$, $Q\in\cD$, we have
\[
\int \Big( \sum_{Q\in\cD} \frac{\lambda_{Q}}{\sigma(Q)} 1_{Q}(x) \Big)^{s} \dif\sigma(x)
\lesssim_{s}
\sum_{Q\in\cD} \lambda_{Q} \Big( \sigma(Q)^{-1} \sum_{Q'\subseteq Q} \lambda_{Q'} \Big)^{s-1}.
\]
\end{lemma}

The following observation seems to come from \cite[Lemma 5.2]{MR3204859}.
The multilinear version appears in \cite[Proposition 4.8]{MR3850284}.
\begin{lemma}
\label{lem:sum<1}
Let $0\leq \beta_{1},\dots,\beta_{m}$ be such that $\beta:=\sum_{i=1}^{m}\beta_{i}<1$.
Then for every cube $Q$, every Carleson sequence $\tau$, and all functions $w_{1},\dots,w_{m}$ we have
\[
\sum_{Q'\subseteq Q, Q'\in\cD} \tau_{Q'} \meas{Q'} \prod_{i=1}^{m} (w_{i})_{Q'}^{\beta_{i}}
\lesssim
\norm{\tau}_{\Car} \meas{Q} \prod_{i=1}^{m} (w_{i})_{Q}^{\beta_{i}}.
\]
\end{lemma}
\begin{proof}
By definition of an infinite sum it suffices to prove the estimate for finite dyadic grids with a constant independent of the cardinality of the grid.
Assume that $\cD$ is finite and apply Lemma~\ref{lem:carleson-sparse} with $\Lambda=\norm{\tau}_{\Car}$ to the sequence $\tau$.
With the sets $E(Q)$ from that lemma we can write the left-hand side of the conclusion as
\begin{align*}
&=
  \norm{\tau}_{\Car} \sum_{Q'\subseteq Q, Q'\in\cD} \meas{E(Q')} \prod_{i=1}^{m} (w_{i})_{Q'}^{\beta_{i}}\\
&\leq
  \norm{\tau}_{\Car} \int_{Q} \prod_{i=1}^{m} (M(w_{i}1_{Q}))^{\beta_{i}}\\  
&\leq
  \norm{\tau}_{\Car} \prod_{i=1}^{m} \big( \int_{Q} (M(w_{i}1_{Q}))^{\beta} \big)^{\beta_{i}/\beta}.
\end{align*}
In each factor we have the estimate
\begin{multline*}
\int_{Q} (M(w1_{Q}))^{\beta}
\lesssim
\sum_{k\in\Z} 2^{k\beta} \meas{Q\cap \{M(w1_{Q})>2^{k}\}}\\
\lesssim
\meas{Q} \sum_{k\in\Z} 2^{k\beta} \min(1,2^{-k}(w)_{Q})
\lesssim
\max(1/\beta,1/(1-\beta)) \meas{Q} (w)_{Q}^{\beta}
\end{multline*}
by the weak type $(1,1)$ inequality for the maximal function.
\end{proof}

\subsection{Key lemma}
\label{sec:key}
The only place in which the multilinear Muckenhoupt characteristic arises in our arguments is the following result that simplifies and extends \cite[Lemma 4.15]{MR3850284}.
\begin{lemma}
\label{lem:key}
Let $s_{1},\dots,s_{m}\in\R$ with $s_{i}>0$ for $i\neq j$, and $q_{1},\dots,q_{m}>0$ with $q_{j} = 1+s_{j}$ be such that
\[
\sum_{i}s_{i} \leq \sum_{i}q_{i},
\quad
\frac{\sum_{i}s_{i}}{\sum_{i}q_{i}} < \min_{i\neq j} \frac{s_{i}}{q_{i}}.
\]
Then for every $0<\alpha<\infty$ we have
\[
\int \Big( \sum_{Q\in\cD} \tau_{Q} \prod_{i=1}^{m} (w_{i})_{Q}^{s_{i} \alpha} 1_{Q} \Big)^{1/\alpha} \dif w_{j}
\lesssim_{\vec s,\vec q,\alpha}
[\vec w]^{\vec q}
\sum_{Q\in\cD} \tau_{Q} \meas{Q} \prod_{i\neq j} (w_{i})_{Q}^{s_{i} - q_{i}}.
\]
\end{lemma}
\begin{proof}
Then the left-hand side of the conclusion is monotonically decreasing in $\alpha$ (here we use that $\tau$ is $\Set{0,1}$-valued) and the right-hand side does not depend on $\alpha$, so it suffices to consider small $\alpha$, in particular we may assume $s_{i}\alpha+\delta_{ij}>0$ and $\alpha<1$.

It follows from the hypothesis that for sufficiently small $\alpha$ there exists an $\epsilon$ such that
\[
\frac{\alpha \sum_{i}s_{i}}{\sum_{i}q_{i}} < \epsilon \leq \min( \frac{1}{1/\alpha-1}, \min_{i} \frac{\alpha s_{i}+\delta_{ij}}{q_{i}}).
\]
By the assumption $\alpha<1$ and Lemma~\ref{lem:cov} the left-hand side of the conclusion is
\[
\sim \sum_{Q} \tau_{Q} \meas{Q} \prod_{i=1}^{m} (w_{i})_{Q}^{\alpha s_{i}+\delta_{ij}}
\Big(w_{j}(Q)^{-1} \sum_{Q'\subseteq Q} \tau_{Q'} \meas{Q'} \prod_{i=1}^{m} (w_{i})_{Q'}^{\alpha s_{i}+\delta_{ij}} \Big)^{1/\alpha-1}
\]
By definition of $\vec q$-characteristic this is
\[
\leq [\vec w]^{\epsilon (1/\alpha-1) \vec q} \sum_{Q} \tau_{Q} \meas{Q} \prod_{i=1}^{m} (w_{i})_{Q}^{\alpha s_{i}+\delta_{ij}}
\Big(w_{j}(Q)^{-1} \sum_{Q'\subseteq Q} \tau_{Q'} \meas{Q'} \prod_{i=1}^{m} (w_{i})_{Q'}^{\alpha s_{i}+\delta_{ij}-\epsilon q_{i}} \Big)^{1/\alpha-1}
\]
By construction we have $\alpha s_{i}+\delta_{ij}-\epsilon q_{i}\geq 0$ and $\sum_{i}(\alpha s_{i}+\delta_{ij}-\epsilon q_{i})<1$.
Hence by Lemma~\ref{lem:sum<1} the above is
\[
\lesssim
[\vec w]^{\epsilon (1/\alpha-1) \vec q} \sum_{Q} \tau_{Q} \meas{Q} \prod_{i=1}^{m} (w_{i})_{Q}^{\alpha s_{i}+\delta_{ij}}
\Big(w_{j}(Q)^{-1} \meas{Q} \prod_{i=1}^{m} (w_{i})_{Q}^{\alpha s_{i}+\delta_{ij}-\epsilon q_{i}} \Big)^{1/\alpha-1}
\]
\[
=
[\vec w]^{\epsilon (1/\alpha-1) \vec q}
\sum_{Q} \tau_{Q} \meas{Q} \prod_{i=1}^{m} (w_{i})_{Q}^{\delta_{ij}+s_{i} -\epsilon q_{i}(1/\alpha-1)}
\]
By construction we have $1-\epsilon (1/\alpha-1)\geq 0$, and using the definition of $\vec q$-characteristic again we obtain the bound
\[
\leq
[\vec w]^{\vec q}
\sum_{Q} \tau_{Q} \meas{Q} \prod_{i=1}^{m} (w_{i})_{Q}^{\delta_{ij}+s_{i} -q_{i}}.
\]
By hypothesis on $q_{j}$ the conclusion follows.
\end{proof}

\subsection{Stopping times}
Let us now introduce the stopping cube families for the martingale maximal function.
Let $\cD$ be a finite dyadic grid and let $\lambda_{i}:\cD\to [0,\infty)$, $Q\mapsto \lambda_{i,Q}$ be a function that takes a cube to a non-negative real number.
The stopping time $\mathcal{F}_{i}$ is the minimal family of cubes with the following properties:
\begin{enumerate}
\item the maximal members of $\cD$ are contained in $\mathcal{F}_{i}$, and
\item if $F\in\mathcal{F}_{i}$, then every maximal subcube $F'\subset F$ with $\lambda_{i,F'}\geq 2\lambda_{i,F}$ is also a member of $\mathcal{F}_{i}$.
\end{enumerate}
For each cube $Q$ let $\pi_{i}(Q)$ (the parent of $Q$ in the stopping family $\mathcal{F}_{i}$) be the smallest cube with $Q\subseteq \pi_{i}(Q) \in \mathcal{F}_{i}$.
We write $\sum_{F_{1},\dots, F_{m}}$ for the sum running over $F_{i}\in\mathcal{F}_{i}$.
We also write
\[
M\lambda_{i}(x) := \sup_{x\in Q\in \cD} \lambda_{i,Q}.
\]

\begin{corollary}
\label{cor:char}
Let $m\geq 2$, $0<p_{1},\dots,p_{m-1}<\infty$.
Define $\alpha := \sum_{i=1}^{m-1} 1/p_{i}$ and suppose
\[
0
<
q_{i}
:=
s_{i} -
\begin{cases}
1/p_{i},& i<m,\\
1-\alpha,& i=m.
\end{cases}
\]
Then
\begin{multline*}
\Big( \sum_{F_{1},\dots,F_{m-1}} \prod_{i=1}^{m-1} \lambda_{i,F_{i}}^{1/\alpha}
\int
\Big(\sum_{Q : \pi_{i}(Q)=F_{i}} \tau_{Q} 1_{Q} \prod_{i=1}^{m} (w_{i})_{Q}^{s_{i}-\delta_{im}} \Big)^{1/\alpha} dw_{m} \Big)^{\alpha}\\
\lesssim
[\vec w]^{\vec q} [\vec w]_{FW}^{(1/p_{i})_{i\neq m}}
\prod_{i=1}^{m-1} \norm{ M\lambda_{i} }_{L^{p_{i}}(w_{i})}.
\end{multline*}
\end{corollary}
\begin{proof}
We estimate the integral on the left-hand side of the conclusion using Lemma~\ref{lem:key} with
\[
\tilde s_{i} = (s_{i}-\delta_{im})/\alpha,
\quad i\leq m
\]
\[
\tilde q_{i} = q_{i}/\alpha
\]
This is indeed possible, because we have
\[
\sum_{i\leq m} \alpha \tilde q_{i}
=
\sum_{i\leq m} s_{i} - (1-\alpha) - \sum_{i<m}1/p_{i}
=
\sum_{i\leq m} s_{i} -1
=
\sum_{i\leq m} \alpha \tilde s_{i}
\]
so that equality holds in the first inequality in the hypothesis of the lemma, and for $i<m$ we have
\[
\tilde q_{i} < \tilde s_{i},
\]
verifying the second inequality.

We obtain the estimate
\begin{equation}
\label{eq:after-key}
\Big( \sum_{F_{1},\dots,F_{m-1}} \prod_{i=1}^{m-1} \lambda_{i,F_{i}}^{1/\alpha}
[\vec w]^{\tilde q} \sum_{Q : \pi_{i}(Q)=F_{i}} \tau_{Q} \meas{Q} \prod_{i=1}^{m-1} (w_{i})_{Q}^{\tilde s_{i}-\tilde q_{i}} \Big)^{\alpha}
\end{equation}
By Lemma~\ref{lem:carleson-sparse} we can find disjoint measurable subsets $E(Q) \subset Q$ such that
\[
\sum_{Q : \pi_{i}(Q)=F_{i}} \tau_{Q} \meas{Q} \prod_{i=1}^{m-1} (w_{i})_{Q}^{\tilde s_{i}-\tilde q_{i}}
=
\norm{\tau}_{\Car}
\sum_{Q : \pi_{i}(Q)=F_{i}} \meas{E(Q)} \prod_{i=1}^{m-1} (w_{i})_{Q}^{1/(\alpha p_{i})},
\]
and by definition of the multilinear maximal function this is
\[
\leq
\norm{\tau}_{\Car} \sum_{Q : \pi_{i}(Q)=F_{i}} \int_{E(Q)} \cM^{(1/(\alpha p_{i}))_{i\neq m}} (\vec{w})
\leq
\norm{\tau}_{\Car} \int_{F_{1}\cap\dotsb\cap F_{m-1}} \cM^{(1/(\alpha p_{i}))_{i\neq m}} (\vec{w}).
\]
By definition \eqref{eq:def:FW} of the FW characteristic we obtain
\[
\eqref{eq:after-key}
\leq
\norm{\tau}_{\Car} [\vec w]^{\vec q}
\Big( \sum_{F_{1},\dots, F_{m-1}} \prod_{i=1}^{m-1} \lambda_{i,F_{i}}^{1/\alpha}
[\vec w]_{FW}^{(1/(\alpha p_{i}))_{i\neq m}} \int_{F_{1}\cap\dots\cap F_{m-1}} \prod_{i=1}^{m-1} w_{i}^{1/(\alpha p_{i})} \Big)^{\alpha}
\]
\[
=
\norm{\tau}_{\Car} [\vec w]^{\vec q} [\vec w]_{FW}^{(1/p_{i})_{i\neq m}}
\Big( \int \prod_{i=1}^{m-1} \sum_{F_{i}} 1_{F_{i}} \lambda_{i,F_{i}}^{1/\alpha} w_{i}^{1/(\alpha p_{i})} \Big)^{\alpha}.
\]
By Hölder's inequality this is bounded by
\[
\leq
\norm{\tau}_{\Car} [\vec w]^{\vec q} [\vec w]_{FW}^{(1/p_{i})_{i\neq m}}
\prod_{i=1}^{m-1} \Big( \int \bigl( \sum_{F_{i}} 1_{F_{i}} \lambda_{i,F_{i}}^{1/\alpha} \bigr)^{\alpha p_{i}} w_{i} \Big)^{1/p_{i}}.
\]
It remains to observe that by definition of the stopping times we have
\[
\bigl( \sum_{F_{i}} 1_{F_{i}} \lambda_{i,F_{i}}^{1/\alpha} \bigr)^{\alpha p_{i}}
\sim
(M \lambda_{i})^{p_{i}},
\]
since at each point the sum on the left-hand side is geometrically increasing and therefore comparable to the maximal summand.
\end{proof}
Alternatively, one can estimate
\[
\cM^{(1/(\alpha p_{i}))_{i\neq m}} (\vec{w})
\leq
\prod_{i=1}^{m-1} (M w_{i})^{1/(\alpha p_{i})},
\]
and the same argument as above gives
\[
\eqref{eq:after-key}
\leq
\norm{\tau}_{\Car} [\vec w]^{\vec q}
\prod_{i=1}^{m-1} \Big( \int \bigl( \sum_{F_{i}} 1_{F_{i}} \lambda_{i,F_{i}}^{1/\alpha} \bigr)^{\alpha p_{i}} M w_{i} \Big)^{1/p_{i}}.
\]
Since the summands $\lambda_{i,F_{i}}$ contributing to $\sum_{F_{i}}$ at each point are geometrically increasing, this is
\[
\lesssim
\norm{\tau}_{\Car} [\vec w]^{\vec q}
\prod_{i=1}^{m-1} \Big( \int \sum_{F_{i}} 1_{F_{i}} \lambda_{i,F_{i}}^{p_{i}} M w_{i} \Big)^{1/p_{i}}.
\]
By definition of the linear FW characteristic this is
\[
\leq
\norm{\tau}_{\Car} [\vec w]^{\vec q}
\prod_{i=1}^{m-1} \Big( [w_{i}]_{FW} \int \sum_{F_{i}} 1_{F_{i}} \lambda_{i,F_{i}}^{p_{i}} w_{i} \Big)^{1/p_{i}}.
\]
This yields an estimate similar to Corollary~\ref{cor:char} with $[\vec w]_{FW}^{(1/p_{i})_{i\neq m}}$ replaced by $\prod_{i\neq m} [w_{i}]_{FW}^{1/p_{i}}$.
It would be interesting to find a joint refinement of these two estimates.

\subsection{Case distinction}
Corollary~\ref{cor:char} can be applied both to the sparse forms $\sform_{\vec\rho}^{\vec r}$ inside the duality range (when $\alpha<1$) and to the sparse operators $\sop_{\vec\rho}^{\vec r}$ outside of the duality range (when $\alpha\geq 1$).
In this section we formulate these results in terms of general sequences $\lambda_{i}$, they will be replaced by suitable means of functions $f_{i}$ in the next section.
\begin{theorem}
\label{thm:convex-range}
Let $m\geq 2$ and $1<p_{i}<\infty$, $i=1,\dots,m$, be such that $\sum_{i=1}^{m} 1/p_{i}=1$ and suppose $q_{i}=s_{i}-1/p_{i}>0$.
Let $\{1,\dots,m\}=J_{s}\cup J_{r}$ be a partition with $J_{s}$ non-empty.
Then
\begin{equation}
\label{eq:thm:convex-range}
\sum_{Q\in \cD} \tau_{Q} \meas{Q} \prod_{i=1}^{m} \lambda_{i,Q} (w_{i})_{Q}^{s_{i}}
\lesssim
[\vec w]^{\vec q} \sum_{j\in J_{s}} [\vec w]_{FW}^{(1/p_{i})_{i\neq j}}
\prod_{i\in J_{s}} \norm{ M\lambda_{i} }_{L^{p_{i}}(w_{i})}
\prod_{i\in J_{r}} \norm{ M\lambda_{i} }_{L^{p_{i},1}(w_{i})}.
\end{equation}
\end{theorem}
The restriction of the sum of FW characteristics to $J_{s}$ is related to the difference between local testing conditions for weak and strong type inequalities in \cite{arXiv:0911.3437}.
\begin{proof}
For the indices $i\in J_{r}$ we may assume that the function $\lambda_{i}$ is $\{0,1\}$-valued, and then restrict the sum to the set of cubes on which it takes the value $1$.
Having done that, we replace the $L^{p_{i},1}$ norm by the $L^{p_{i}}$ norm and construct stopping times as above.
The left-hand side of the conclusion is estimated by
\[
\lesssim
\sum_{F_1,\dots,F_{m}} \sum_{Q : \pi_{i}(Q)=F_{i}} \tau_{Q} \meas{Q} \prod_{i=1}^{m} \lambda_{i,F_{i}} (w_{i})_{Q}^{s_{i}}.
\]
The sum over $Q$ can only be non-trivial if $F_{1},\dots,F_{m}$ intersect non-trivially and the smallest set(s) $F_{j}$ among them satisfies $\pi_{i}(F_{j})=F_{i}$ for all $i$.
Since the functions $\lambda_{i}$, $i\in J_{r}$, are constant, this smallest set always comes from $\cF_{j}$ with $j\in J_{s}$.
Thus we can estimate
\begin{equation}
\label{eq:convex-range:split-by-minimal}
\sum_{F_1,\dots,F_{m}}
\leq
\sum_{j\in J_{s}} \ \sum_{F_{i}, i\neq j} \ \sum_{F_{j} : \pi_{i}(F_{j})=F_{i}}.
\end{equation}
By symmetry it suffices to consider one of these terms, and for simplicity of notation we assume $m\in J_{s}$ and consider $j=m$.
We estimate the contribution of that term by
\[
\leq
\int \sum_{F_{1},\dots,F_{m-1}} \Big(\sum_{F_{m}} 1_{F_{m}}\lambda_{m,F_{m}}\Big) \prod_{i=1}^{m-1} (1_{F_{i}} \lambda_{i,F_{i}})
\sum_{Q : \pi_{i}(Q)=F_{i}, i<m} \tau_{Q} 1_{Q} \prod_{i=1}^{m} (w_{i})_{Q}^{s_{i}-\delta_{im}} dw_{m}
\]
applying H\"older with the pair of conjugate exponents $p_{m}$ and $1/\alpha$, where $\alpha = \sum_{i=1}^{m-1} 1/p_{i}$,
we obtain
\begin{multline*}
\leq
\Big( \int \sum_{F_{1},\dots, F_{m-1}} \Big( \sum_{F_{m}} 1_{F_{m}}\lambda_{m,F_{m}} \Big)^{p_{m}} dw_{m} \Big)^{1/p_{m}}\\\cdot
\Big( \int \sum_{F_{1},\dots, F_{m-1}} \prod_{i=1}^{m-1} \lambda_{i,F_{i}}^{1/\alpha}
\Big(\sum_{Q : \pi_{i}(Q)=F_{i}} \tau_{Q} 1_{Q} \prod_{i=1}^{m} (w_{i})_{Q}^{s_{i}-\delta_{im}} \Big)^{1/\alpha} dw_{m} \Big)^{\alpha}
\end{multline*}
The first term is estimated by $\norm{ M\lambda_{m} }_{L^{p_{m}}(w_{m})}$.
In the second term we apply Corollary~\ref{cor:char} and note that $s_{m}-(1-\alpha)=s_{m}-1/p_{m}$.
\end{proof}

\begin{theorem}
\label{thm:concave-range}
Let $m\geq 2$ and $0<p_{i},s_{i}<\infty$, $1\leq i<m$, and let $\alpha:=\sum_{i=1}^{m-1} 1/p_{i}$.
Suppose $q_{i}:=s_{i}-1/p_{i}>0$ for $i<m$ and let $q_{m}:=\alpha$.
\begin{enumerate}
\item
If $\alpha\geq 1$, then
\begin{equation}
\label{eq:concave-range}
\norm[\Big]{ \sum_{Q\in \cD} \tau_{Q} \prod_{i=1}^{m-1} \lambda_{i,Q} (w_{i})_{Q}^{s_{i}} 1_{Q} }_{L^{1/\alpha}(w_{m})}
\lesssim
[\vec w]^{\vec q} [\vec w]_{FW}^{(1/p_{i})_{i\neq m}}
\prod_{i=1}^{m-1} \norm{ M\lambda_{i} }_{L^{p_{i}}(w_{i})}.
\end{equation}
\item
Let $\tilde E(Q)\subset Q$, $Q\in\cD$, be disjoint subsets.
Then
\begin{equation}
\label{eq:disjoint-supp}
\norm[\Big]{ \sum_{Q\in \cD} \tau_{Q} \prod_{i=1}^{m-1} \lambda_{i,Q} (w_{i})_{Q}^{s_{i}} 1_{\tilde E(Q)} }_{L^{1/\alpha}(w_{m})}
\lesssim
[\vec w]^{\vec q} [\vec w]_{FW}^{(1/p_{i})_{i\neq m}}
\prod_{i=1}^{m-1} \norm{ M\lambda_{i} }_{L^{p_{i}}(w_{i})}.
\end{equation}
\end{enumerate}
\end{theorem}
\begin{proof}[Proof of \eqref{eq:concave-range}]
The left-hand side of the conclusion can be estimated by
\[
\Big( \int \Big( \sum_{F_{1},\dots,F_{m-1}} \prod_{i=1}^{m-1} \lambda_{i,F_{i}}
\sum_{Q : \pi_{i}(Q)=F_{i}} \tau_{Q} \prod_{i=1}^{m-1}(w_{i})_{Q}^{s_{i}} 1_{Q} \Big)^{1/\alpha} dw_{m} \Big)^{\alpha}.
\]
By subadditivity of the function $x\mapsto x^{1/\alpha}$ this is bounded by
\[
\Big( \sum_{F_{1},\dots, F_{m-1}} \prod_{i=1}^{m-1} \lambda_{i,F_{i}}^{1/\alpha}
\int \Big( \sum_{Q : \pi_{i}(Q)=F_{i}} \tau_{Q} \prod_{i=1}^{m-1}(w_{i})_{Q}^{s_{i}} 1_{Q} \Big)^{1/\alpha} dw_{m} \Big)^{\alpha}.
\]
Applying Corollary~\ref{cor:char} with $s_{m}=1$ we obtain the claim.
\end{proof}
\begin{proof}[Proof of \eqref{eq:disjoint-supp}]
Once again, the left-hand side of the conclusion can be estimated by
\[
\Big( \int \Big( \sum_{F_{1},\dots,F_{m-1}} \prod_{i=1}^{m-1} \lambda_{i,F_{i}} \sum_{Q:\pi_{i}(Q)=F_{i}} \tau_{Q} \prod_{i=1}^{m-1} (w_{i})_{Q}^{s_{i}} 1_{\tilde E(Q)} \Big)^{1/\alpha} dw_{m} \Big)^{\alpha}.
\]
Now we use disjointness of the $\tilde E(Q)$'s to pull the sum over $F_{1},\dots,F_{m-1}$ out of the power $1/\alpha$ and then out of the integral, this gives the estimate
\[
\Big( \sum_{F_{1},\dots,F_{m-1}} \prod_{i=1}^{m-1} \lambda_{i,F_{i}}^{1/\alpha} \int \Big( \sum_{Q:\pi_{i}(Q)=F_{i}} \tau_{Q} \prod_{i=1}^{m-1} (w_{i})_{Q}^{s_{i}} 1_{\tilde E(Q)} \Big)^{1/\alpha} dw_{m} \Big)^{\alpha}.
\]
At this point we estimate $1_{\tilde E(Q)}\leq 1_{Q}$ and apply Corollary~\ref{cor:char} with $s_{m}=1$.
\end{proof}

\section{Applications}
\subsection{Sparse forms}
\begin{proof}[Proof of Theorem~\ref{thm:fractional-dualized}]
With the choice
\[
\lambda_{i,Q} := \Big( w_{i}(Q)^{-(1-\rho_{i})} \int_{Q} f_{i} w_{i} \Big)^{r_{i}}
\]
the left-hand side of the conclusion can be written in the form \eqref{eq:thm:convex-range} with $s_{i}=(1-\rho_{i})r_{i}$.
We apply Theorem~\ref{thm:convex-range} with the exponents $\tilde p_{i}$ given by
\[
\frac{1}{r_{i}\tilde p_{i}} = \frac{1}{t_{i}} - \rho_{i}
\]
and obtain the estimate
\[
[\vec w]^{\vec q} \sum_{j\in J_{s}} [\vec w]^{(1/\tilde p_{i})_{i\neq j}}
\prod_{i\in J_{s}} \norm{M\lambda_{i}}_{L^{\tilde p_{i}}(w_{i})} \prod_{i\in J_{r}} \norm{M\lambda_{i}}_{L^{\tilde p_{i},1}(w_{i})}.
\]
We compute
\[
\norm{M\lambda_{i}}_{L^{\tilde p_{i},s}(w_{i})}
\leq
\norm{(M_{\rho_{i},w_{i}}f_{i})^{r_{i}}}_{L^{\tilde p_{i},s}(w_{i})}
=
\norm{M_{\rho_{i},w_{i}}f_{i}}_{L^{r_{i} \tilde p_{i},r_{i} s}(w_{i})}^{r_{i}}
\lesssim
\norm{f_{i}}_{L^{t_{i},r_{i} s}(w_{i})}^{r_{i}},
\]
where $M_{\rho_{i},w_{i}}$ is the martingale fractional maximal function with respect to the reference measure $w_{i}d\mu$ and the last inequality is given by Proposition~\ref{prop:frac-max}.
Using this with $s=1$ for $i\in J_{r}$ and $s=\tilde p_{i}\geq t_{i}/r_{i}$ for $i\in J_{s}$ we obtain the claim.
\end{proof}
Theorem~\ref{thm:fractional-below-L1} is proved in the same way using \eqref{eq:concave-range} instead of Theorem~\ref{thm:convex-range}.

\subsubsection{Bilinear Hilbert transform}
One of the motivating examples for Theorem~\ref{thm:fractional-dualized} is the following sparse domination result for the bilinear Hilbert transform due to Culiuc, di Plinio, and Ou (in fact their result holds for more general trilinear forms, the reader is invited to consult their article for the definitions).
\begin{theorem}[{\cite[Theorem 2]{arXiv:1603.05317}}]
\label{thm:bht-sparse}
Let $\Lambda$ be the trilinear form adjoint to the bilinear Hilbert transform.
Consider exponents
\[
1< p_{1},p_{2},p_{3} < \infty
\text{ with }
\sum_{j=1}^{3} \frac{1}{\min(p_{j},2)} < 2.
\]
Then for any tuple $(g_{1},g_{2},g_{3})$ of smooth functions with compact support on $\R$ there exists a $1/6$-sparse collection $\sparse$ of intervals in $\R$ such that
\[
\abs{\Lambda(\vec g)}
\lesssim
\sum_{I\in\sparse} \meas{I} \prod_{j=1}^{3} \Big( \meas{I}^{-1} \int_{I} \abs{g_{j}}^{p_{j}} )^{1/p_{j}}.
\]
\end{theorem}
\begin{corollary}
\label{cor:bht-mixed-char}
Let $1<q_{1},q_{2},q_{3}<\infty$ with $\sum_{j=1}^{3}1/q_{j}=1$ and let $p_{j}<q_{j}$ be as in Theorem~\ref{thm:bht-sparse}.
Then
\[
\abs{\Lambda(\widevec{fw})}
\lesssim
[(w_{i})^{p_{i}}]^{(1/p_{i}-1/q_{i})_{i}} \sum_{j} [(w_{i})^{p_{i}}]_{FW}^{(1/q_{i})_{i\neq j}} \prod_{i=1}^{3} \norm{f_{i}}_{L^{q_{i}}(w_{i}^{p_{i}})}.
\]
\end{corollary}
\begin{proof}
Apply Theorem~\ref{thm:fractional-dualized} with $m=3$, $J_{r}=\emptyset$, functions $g_{i}=\abs{f_{i}}^{p_{i}}$, weights $v_{i}=w_{i}^{p_{i}}$, and exponents $r_{i}=1/p_{i}$, $t_{i}=q_{i}/p_{i}$, $\rho_{i}=0$.
\end{proof}
Since $\Lambda$ is a multilinear operator, another natural formulation of Corollary~\ref{cor:bht-mixed-char} is
\[
\abs{\Lambda(\widevec{fu})}
\lesssim
[(u_{i})^{p_{i}\frac{q_{i}-1}{q_{i}-p_{i}}}]^{(1/p_{i}-1/q_{i})_{i}} \sum_{j} [(u_{i})^{p_{i}\frac{q_{i}-1}{q_{i}-p_{i}}}]_{FW}^{(1/q_{i})_{i\neq j}} \prod_{i=1}^{3} \norm{f_{i}}_{L^{q_{i}}(u_{i})},
\]
which can be obtained from the above formulation by substituting $u_{i} = w_{i}^{\frac{q_{i}-p_{i}}{q_{i}-1}}$.

In the case of dependent weights $\prod_{i=1}^{3} (w_{i}^{p_{i}})^{1/p_{i}-1/q_{i}}=1$ Corollary~\ref{cor:bht-mixed-char} recovers \cite[Theorem 3]{arXiv:1603.05317}.
Indeed, by Lemma~\ref{lem:FW-Ap} we obtain
\[
\abs{\Lambda(\widevec{fw})} \lesssim [(w_{i})^{p_{i}}]^{(1/p_{i}-1/q_{i})_{i} \max_{j} \frac{q_{j}}{q_{j}-p_{j}}} \prod_{i=1}^{3} \norm{f_{i}}_{L^{q_{i}}(w_{i}^{p_{i}})}.
\]
The $i$-th norm on the right-hand side can be written as $\norm{f_{i}w_{i}}_{L^{q_{i}}(v_{i})}$ with $v_{i}=w_{i}^{p_{i}-q_{i}}$.
The weights $v_{i}$ then satisfy the relation $\prod_{i} v_{i}^{1/q_{i}} = 1$ and we have the weighted norm estimate
\[
\abs{\Lambda(\widevec{f})} \lesssim [(v_{i})^{\frac{p_{i}}{p_{i}-q_{i}}}]^{(1/p_{i}-1/q_{i})_{i} \max_{j} \frac{q_{j}}{q_{j}-p_{j}}} \prod_{i=1}^{3} \norm{f_{i}}_{L^{q_{i}}(v_{i})},
\]
which is \cite[Theorem 3]{arXiv:1603.05317}.

\subsection{Weak type maximal inequality}
\label{sec:weak-max}
\begin{proof}[Proof of \eqref{eq:max-weak}]
The set $\{\cM^{\vec r}_{\vec\rho}(\widevec{fw}) > \lambda \}$ is the disjoint union of the maximal cubes contained in it.
Call the family of these maximal cubes $\sparse$.
We have
\[
\lambda^{1/\alpha} w_{m} \{\cM^{\vec r}_{\vec\rho}(\widevec{fw}) > \lambda \}
\leq
\sum_{Q\in\sparse} \prod_{i=1}^{m-1} \Big(\meas{Q}^{-(1-\rho_{i})} \int_{Q} f_{i}w_{i} \Big)^{r_{i}/\alpha} w_{m}(Q).
\]
Applying Hölder's inequality in the inner integral we obtain the estimate
\[
\leq
\sum_{Q\in\sparse} w_{m}(Q) \prod_{i=1}^{m-1} \Big(\meas{Q}^{-(1-\rho_{i})} \big(\int_{Q} f_{i}^{t_{i}}w_{i}\big)^{1/t_{i}} \big( \int_{Q} w_{i} \big)^{1/t_{i}'} \Big)^{r_{i}/\alpha}
\]
\[
=
\sum_{Q\in\sparse} (w_{m})_{Q} \prod_{i=1}^{m-1} \Big(\big(\int_{Q} f_{i}^{t_{i}}w_{i}\big)^{1/t_{i}} ( w_{i} )_{Q}^{1/t_{i}'} \Big)^{r_{i}/\alpha}
\]
by definition of the weight characteristic this is
\[
\leq [\vec w]^{\vec q/\alpha}
\sum_{Q\in\sparse} \prod_{i=1}^{m-1} \Big(\int_{Q} f_{i}^{t_{i}}w_{i}\Big)^{r_{i}/(t_{i} \alpha)}
\]
By Hölder's inequality with exponents $(1/t_{i}-\rho_{i})^{-1} \alpha/r_{i}$ this is
\[
\leq [\vec w]^{\vec q/\alpha}
\prod_{i=1}^{m-1} \Big( \sum_{Q\in\sparse} \big(\int_{Q} f_{i}^{t_{i}}w_{i}\big)^{(1/t_{i}-\rho_{i})^{-1}/t_{i}} \Big)^{(1/t_{i}-\rho_{i})r_{i}/\alpha}.
\]
We can view each sum as a little $\ell^{x}$ norm with $x>1$ and estimate it by the corresponding $\ell^{1}$ norm.
This gives the estimate
\[
\leq [\vec w]^{\vec q/\alpha}
\prod_{i=1}^{m-1} \Big( \sum_{Q\in\sparse} \int_{Q} f_{i}^{t_{i}}w_{i} \Big)^{r_{i}/(t_{i} \alpha)}.
\]
Since the sets $Q$ are pairwise disjoint we have obtained the required estimate for the $1/\alpha$-th power of the left-hand side of the conclusion.
\end{proof}

\subsection{Strong type maximal inequality}
\label{sec:max-strong}
\begin{proof}[Proof of \eqref{eq:max-strong}]
By the monotone convergence theorem we may assume that $\cD$ is finite.
We construct the stopping family $\sparse$ as follows.
Firstly, we put the maximal elements of $\cD$ into $\sparse$.
If $Q\in\sparse$, then we add to $\sparse$ all maximal subcubes $Q'$ of $Q$ on which the argument of the supremum in \eqref{eq:mult-frac-max} exceeds its value on $Q$ at least by the factor $(2m)^{\sum_{i} r_{i}}$.
Then in particular
\[
2m \meas{Q}^{-(1-\rho_{i})} \int_{Q} f_{i}w_{i}
\leq
\meas{Q'}^{-(1-\rho_{i})} \int_{Q'} f_{i}w_{i}
\]
for some $i$.
We write this as
\[
2m \meas{Q'} \int_{Q} f_{i}w_{i}
\leq
\meas{Q} (\meas{Q'}/\meas{Q})^{\rho_{i}} \int_{Q'} f_{i}w_{i}
\leq
\meas{Q} \int_{Q'} f_{i}w_{i}.
\]
Summing this inequality over all cubes $Q'$ we see that the stopping collection $\sparse$ is sparse.
Moreover,
\[
\cM^{\vec r}_{\vec\rho}(\widevec{fw})
\lesssim
\sum_{Q\in \sparse} \prod_{i=1}^{m-1} \lambda_{i,Q} (w_{i})_{Q}^{s_{i}} 1_{\tilde E(Q)}
\]
with disjoint subsets $\tilde E(Q)\subset Q$ and notation from Theorem~\ref{thm:fractional-dualized}.
We conclude using \eqref{eq:disjoint-supp}.
\end{proof}

\section{Dependent weights}
\label{sec:dependent}

\begin{proof}[Proof of Lemma~\ref{lem:FW-Ap}]
By homogeneity we may assume $\sum_{i} \beta_{i} = \sum_{i}q_{i} = 1$.

Let $Q_{0}$ be a dyadic cube, by a stopping time argument we can find a sparse collection of subcubes $Q\subseteq Q_{0}$ with disjoint major subsets $E(Q)\subset Q$ (``major'' means that $\meas{Q}\lesssim \meas{E(Q)}$) such that
\[
\int_{Q_{0}} \mathcal{M}^{\vec\beta}(1_{Q_{0}}\vec w_{(j)})
\lesssim
\sum_{Q} \meas{Q} \prod_{i} (w_{i})_{Q}^{\beta_{i}}.
\]
This can be written as
\[
\sum_{Q} \Big( \int_{E(Q)} \prod_{i} w_{i}^{\beta_{i}} \Big)
\cdot \underbrace{\meas{Q} \frac{\prod_{i} (w_{i})_{Q}^{\beta_{i}} w_{i}(E(Q))^{\gamma q_{i}-\beta_{i}}}%
{\prod_{i} w_{i}(E(Q))^{\gamma q_{i}-\beta_{i}} \int_{E(Q)} \prod_{i} w_{i}^{\beta_{i}}}}_{*},
\]
and it suffices to show $*\lesssim [\vec w]^{\gamma \vec q}$ uniformly in $Q$.

Note that $\gamma q_{i}-\beta_{i}\geq 0$ for all $i$, so by H\"older's inequality we have
\[
\meas{E(Q)}
=
\int_{E(Q)} \prod_{i=1}^{m} w_{i}^{q_{i}}
=
\int_{E(Q)} \prod_{i=1}^{m} w_{i}^{\frac{ \gamma q_{i}-\beta_{i}}{\gamma}} \prod_{i=1}^{m} w_{i}^{\frac{\beta_{i}}{\gamma}}
\leq
\prod_{i=1}^{m} \Big( \int_{E(Q)} w_{i} \Big)^{\frac{ \gamma q_{i}-\beta_{i}}{\gamma}}
\cdot
\Big( \int_{E(Q)} \prod_{i=1}^{m} w_{i}^{\beta_{i}} \Big)^{\frac{1}{\gamma}}.
\]
Raising this inequality to the power $\gamma$ and applying it in the denominator we obtain
\[
* \leq \meas{Q} \meas{E(Q)}^{-\gamma} \prod_{i} (w_{i})_{Q}^{\beta_{i}} w_{i}(E(Q))^{\gamma q_{i}-\beta_{i}}.
\]
Using positivity of the exponents $\gamma q_{i}-\beta_{i}$ and the fact that $E(Q)\subset Q$ are major subsets, we obtain the estimate
\[
* \lesssim
\meas{Q}^{1-\gamma} \prod_{i} (w_{i})_{Q}^{\beta_{i}} w_{i}(Q)^{\gamma q_{i}-\beta_{i}}
=
\prod_{i} (w_{i})_{Q}^{\gamma q_{i}}
\leq
[\vec w]^{\gamma \vec q}.
\qedhere
\]
\end{proof}

\printbibliography
\end{document}